%% file: coverings-by-graphs-bounded-by-certain-parameters.tex
\documentclass[]{article}

\title{Cover numbers by graph families bounded by certain graph parameters}

\author{Anna Gujgiczer\thanks{HUN-REN Alfr\'ed R\'enyi Institute of Mathematics, Budapest, Hungary} \thanks{MTA-HUN-REN RI Lend\"ulet ``Momentum" Arithmetic Combinatorics Research Group, Budapest, Hungary} \and Márton Marits\thanks{Yokohama National University, Yokohama, Japan} \and Kenta Ozeki\thanks{Yokohama National University, Yokohama, Japan}}

\date{\today}
\usepackage[yyyymmdd]{datetime}

\usepackage{amsmath}
\usepackage{amsthm, amssymb}
\usepackage{graphicx}

\usepackage{xcolor}
\usepackage{tikz}
\usetikzlibrary{calc}

\usepackage{caption}
\usepackage{subcaption}
\usepackage{subcaption}
\usepackage{float}

\newcommand{\N}{\mathbb{N}}

\newcommand{\Q}{\mathbb{Q}}
\newcommand{\Z}{\mathbb{Z}}

\usepackage{bm}

\newcommand{\eps}{\varepsilon}
\renewcommand{\rho}{\varrho}
\renewcommand{\phi}{\varphi}

\let\existstemp\exists
\let\foralltemp\forall
\renewcommand*{\exists}{\ \existstemp}
\renewcommand*{\forall}{\ \foralltemp}

\theoremstyle{plain}
\newtheorem{thm}{Theorem}[section]

\newtheorem{cor}[thm]{Corollary}

\newtheorem{question}[thm]{Question}

\usepackage{hyperref}

\usepackage{fontspec}

\begin{document}

\maketitle

\begin{abstract}
	
	The \textit{cover number} of a graph by a graph class $\mathcal P$ is the least number of $\mathcal P$-graphs necessary to cover its edges. A classical theorem of Harary, Hsu and Miller gives an exact formula for the cover number by the class of graphs with chromatic number at most $k$. We investigate analogous questions for the case of the fractional chromatic number $\chi_f$ and the local chromatic number $\psi$.
	
	We prove that an analogous formula cannot hold in the case of the cover number by graphs of fractional chromatic number at most $\beta$, and find a lower and an upper bound, that gives rise to interesting asymptotic questions. We also investigate this cover number for small specific graphs. In the case of the cover number by graphs with local chromatic number at most $k$, we find an upper bound in terms of $\psi$, and a lower bound in terms of $\omega$. 
	
\end{abstract} 


\section{Introduction} \label{sec:intro}



A \textit{covering} of a graph $G$ is a set of subgraphs of $G$ such that each edge of $G$ is contained in at least one of these subgraphs. The \textit{cover number} of $G$ by a graph class $\mathcal P$ is the smallest number $k$ such that $G$ can be covered by $k$ subgraphs belonging to $\mathcal P$. 


A \textit{proper coloring} of a graph $G$ with $n$ colors is a function $V(G) \to \{1, \dots, n\}$ that maps neighboring vertices of $G$ to different colors. The \textit{chromatic number} $\chi(G)$ of a graph $G$ is the smallest $n$ such that it admits a proper coloring with $n$ colors. A \textit{clique} of size $n$ in a graph $G$ is subgraph of $G$ isomorphic to the complete graph $K_n$. The \textit{clique number} $\omega(G)$ of a graph $G$ is the largest $n$ such that it contains a clique of size $n$. We delay other definitions until Section \ref{sec:prelim}.

The starting point of this research is the following theorem by Harary, Hsu and Miller \cite{HararyHsuMiller1977}. The notation $\{\chi \leq k\}$ here denotes the set of graphs with chromatic number at most $k$, i.e. it is an abbreviation of the set $\{ G \mid \chi(G) \leq k\}$. Similarly, in the rest of this paper, we will consider cover numbers by several graph classes of the form $\mathcal P = \{G \mid L(G)\}$, where $L$ is a formula that contains constants and graph parameters. Such graph classes will always be abbreviated as $\{L\}$, following the notation of \cite{Marits2026}.

\begin{thm}[Harary, Hsu, Miller 1977 \cite{HararyHsuMiller1977}] \label{thm:hhm}
	For all graphs $G$, and all positive integers $k$ we have \[
		c_{\{\chi \leq k\}}(G) = \left\lceil \frac{\log \chi(G)}{\log k} \right\rceil
	\]
\end{thm}

This theorem gives an exact formula for the cover number $c_{\{\chi \leq k\}}(G)$ in terms of $\chi(G)$. Recently, a similar exact formula has been found in \cite{Marits2026}, replacing the constant $k$ by the clique number $\omega$. The following can be said about the graphs class $\{\chi = \omega\}$. This graph class is closely related to the well-researched family of perfect graphs.

\begin{thm}[Marits 2026 \cite{Marits2026}] \label{thm:m-omega}
	For all graphs $G$, we have \[
		c_{\{\chi = \omega\}}(G) = \left\lceil \frac{\log \chi(G)}{\log \omega(G)} \right\rceil
	\]
\end{thm}


In light of this result, our goal is to find analogs of Theorem \ref{thm:hhm}, by replacing the ordinary chromatic number with other graph parameters. We will consider the fractional chromatic number $\chi_f$ and the local chromatic number $\psi$, both of which we will define in Section \ref{sec:prelim}.

We have obtained results for the cases of both the fractional and the local chromatic number. Results with the fractional chromatic number are explored in Section \ref{sec:fractional}. The fractional chromatic number can take non-integer values, therefore, we investigate the graph class $\{\chi_f \leq \beta\}$, for some rational number $\beta$. We find the following upper and lower bound.

\begin{thm} \label{thm:INTRO_frac-upper-lower}
	Let $G$ be any graph, and $\beta > 2$ a rational number. Then \[
			\left\lceil \frac{\log \chi_f(G)}{\log \beta} \right\rceil \leq c_{\{\chi_f \leq \beta\}}(G) \leq \left\lceil \frac{\log \chi(G)}{\log \omega(G)} \right\rceil \cdot \left\lceil \frac{\log\chi_f(G)}{\log \left\lfloor \beta \right\rfloor}  \right\rceil
	\]
\end{thm}

This theorem is a partial analog of Theorem \ref{thm:hhm}, in fact, in the lower bound we find an exact analog of the formula from that theorem. However, we find that there exist graphs whose cover number by $\{\chi_f \leq \beta\}$ must be strictly bigger than this lower bound. In Section \ref{sec:specificcases}, we investigate the graph parameter $c_{\{\chi_f \leq \beta\}}(G)$ for specific, small choices of $\beta$ and $G$. We conclude that no general formula containing only $\chi_f(G)$ can exist for $c_{\{\chi_f \leq \beta\}}(G)$. 

In Theorem \ref{thm:INTRO_frac-upper-lower}, we have a lower bound on $c_{\{\chi_f \leq \beta\}}$ containing the fractional chromatic number, and an upper bound containing both the fractional chromatic number and the ordinary chromatic number. This means that our bounds are not tight for graph families where the fractional and ordinary chromatic numbers are far from each other. We will investigate two families of this type in Section \ref{sec:eightconj}.

Regarding the local chromatic number, we investigate the graph class $\{\psi \leq k\}$, for some integer $k$. We find the following upper and lower bound. Detailed information on the local chromatic number $\psi$ and the graph parameter $c_{\{\psi \leq k\}}$ is in Section \ref{sec:local}.

\begin{thm} \label{thm:INTRO_local-upper}
	For all graphs $G$ and integers $k \geq 3$, we have \[
		f_k^{-1}(\omega(G)) \leq c_{\{\psi \leq k\}}(G) \leq \left\lceil \frac{\psi(G) - 1}{\left\lfloor \frac{k-1}{2} \right \rfloor} \right\rceil \cdot \left\lceil \frac{\psi(G) - 1}{\left\lceil \frac{k-1}{2} \right \rceil} \right\rceil
	\] where $f_k(x)$ is a function of the order $O(x! \cdot k^x)$.
\end{thm}

The upper bound of \ref{thm:INTRO_local-upper} is analogous to Theorem \ref{thm:hhm} in the sense that it only depends on $\psi(G)$ -- the upper bound on the cover number depends on the same graph parameter as the condition on the covering graphs. The lower bound, however, does not depend on $\psi(G)$ directly, only on $\omega(G)$.

Theorems \ref{thm:INTRO_frac-upper-lower} and \ref{thm:INTRO_local-upper} give rise to interesting asymptotic questions, as the upper and lower bounds are asymptotically different. It is well-known from the Lovász-Kneser theorem \cite{Lovasz1978} that graphs with constant fractional chromatic number can attain arbitrarily large values for their ordinary chromatic number. For a family of graphs $G_n$ with the property that $\chi_f(G_n)$ is constant and $\chi(G_n) \to \infty$, the upper bound of Theorem \ref{thm:INTRO_frac-upper-lower} goes to infinity, while the lower bound is a constant. Therefore, in Section \ref{sec:eightconj}, we investigate whether $c_{\{\chi_f \leq \beta\}}(G_n)$ goes to infinity, with two natural choices for $G_n$. We also investigate the behavior of other relevant cover numbers for our choices of $G_n$.

The choices of $G_n$ we consider are the following. First, the sequence $KG(\alpha k, k)$ of Kneser graphs with constant fractional chromatic number $\alpha \in \Q$, and second, the sequence $U(k, \alpha)$ with constant local chromatic number (and constant fractional chromatic number) $\alpha \in \N$. The precise definitions of these sequences of graphs are delayed until later sections.

\section{Preliminaries} \label{sec:prelim}

All graphs in this work are undirected and simple. 
A $G \to H$ \textit{homomorphism} between the graphs $G$ and $H$ is a function $\phi\colon V(G) \to V(H)$ that maps each edge of $G$ to an edge of $H$. We write $G \to H$ if such a homomorphism exists. All logarithms are to the base two unless stated otherwise.

Let $G$ be a graph, and let $\mathcal P$ be a family of graphs. We say that $G$ is \textit{covered} by the graphs $G_1, \dots, G_c \in \mathcal P$ if $E(G) = E(G_1) \cup \dots \cup E(G_c)$ and each $G_i$ is a subgraph of $G$. We call a covering a \textit{partition} or \textit{decomposition} if the covering graphs are pairwise edge-disjoint. Most choices of $\mathcal P$ we consider in this work are closed with respect to taking subgraphs, in which case we can assume any covering by $\mathcal P$-graphs to be a partition. All choices of $\mathcal P$ in this work are closed with respect to the addition or removal of isolated vertices. Because of this, we can always assume that in any covering $E(G) = E(G_1) \cup \dots \cup E(G_c)$, the vertex set of each $G_i$ is $V(G)$. 

The \textit{cover number} of $G$ by $\mathcal P$ is the smallest number $c$ such that there exists a covering of $G$ by $c$ graphs $G_1, \dots, G_c \in \mathcal P$. It is easy to see that for any graph class $\mathcal P$ with $K_2 \in \mathcal P$, this is a well-defined graph parameter. In this paper, all graph classes are assumed to contain $K_2$. We denote the cover number of $G$ by $\mathcal P$ as $c_{\mathcal P}(G)$. Cover numbers were first studied in \cite{Harary1970}, and they are surveyed in \cite{Pyber1991}.

The following assertion is an easy consequence of the definition.

\begin{thm} \label{thm:coverno-bijection}
	Let $\mathcal P, \mathcal Q$ be graph classes. Then the following are equivalent:
	\begin{enumerate}
		\item $\mathcal P \subseteq \mathcal Q$
		\item For all graphs $G$, $c_{\mathcal Q}(G) \leq c_{\mathcal P}(G)$
	\end{enumerate}
\end{thm}

\begin{proof}
	The first assertion implies the second since every covering by $\mathcal Q$-graphs is also a covering by $\mathcal P$-graphs. For the opposite direction, consider that $\mathcal P = \{ G \mid c_{\mathcal P}(G) = 1\}$, which is obviously contained in $\{ G \mid c_{\mathcal Q}(G) = 1\} = \mathcal Q$. (Since all cover numbers are at least one.)
\end{proof}

In our study, we focus on cover numbers by graph classes $\mathcal P$ where membership in $\mathcal P$ is defined by the existence of a homomorphism into a member of some graph family. For example, in Theorem \ref{thm:hhm}, we see the graph class $\{\chi \leq k\}$, which can be thought of as the set of graphs having a homomorphism into the complete graph $K_k$. The following can be said about graph classes $\mathcal P$ defined in this way.

\begin{thm} \label{thm:graph-class-defd-by-homomorphism}
	Let $\mathcal Z$ be a set of graphs, and $\mathcal P := \{ G \mid \exists Z \in \mathcal Z : G \to Z\}$. Then for all graphs $G, H$, if a homomorphism $G \to H$ exists, we have \[
		c_{\mathcal P}(G) \leq c_{\mathcal P}(H)
	\]
\end{thm}

\begin{proof}
	Let $G, H$ be graphs with a homomorphism $\phi\colon V(G) \to V(H)$. Let $H_1, \dots, H_c$ be an optimal covering of $H$ with $c = c_{\mathcal P}(H)$-many graphs in $\mathcal P$. We claim that the subgraphs $G_i \subseteq G$ defined by \[
		E(G_i) := \{ e \in E(G) \mid \phi(e) \in E(H_i)\}
	\] are a valid covering of $G$ with $\mathcal P$-graphs. It is easy to see that this is indeed a covering of $G$, as $\phi$ must map each edge of $G$ to an edge of $H$, which is necessarily covered by at least one of the $H_i$-s.
	
	Suppose now that each $H_i$ has a homomorphism $\tau_i\colon H_i \to Z_i$ for some $Z_i \in \mathcal Z$. Then the composition $\tau_i \circ \phi$ is a homomorphism of $G_i$ to $Z_i$ for each $i$, therefore, each $G_i$ is in $\mathcal P$. This is therefore a covering of $G$ with $c_{\mathcal P}(H)$-many $\mathcal P$-graphs, proving the theorem.
\end{proof}

In this paper, we only study graph classes with the property described in Theorem \ref{thm:graph-class-defd-by-homomorphism}. We will use the term \textit{model graphs of $\mathcal P$} to refer to the elements of $\mathcal Z$ for these choices of $\mathcal P$.

The following corollary will be useful when studying upper bounds on graph classes defined in this way. 

\begin{cor} \label{cor:upper-bound-reasoning}
	Let $\mathcal P$ and $\mathcal Q$ be graph classes, both defined by a class of model graphs, as in Theorem \ref{thm:graph-class-defd-by-homomorphism}. Let $\mathcal W$ be the set of model graphs of $\mathcal Q$. Then \[
		\sup \left\{ c_{\mathcal P}(Q) \mid Q \in \mathcal Q \right\} = \sup \left\{ c_{\mathcal P}(Q) \mid Q \in \mathcal W \right\}
	\]
\end{cor}

In other words, to find an upper bound on $c_{\mathcal P}$ among $\mathcal Q$-graphs, it is sufficient to find an upper bound on the value of $c_{\mathcal P}$ on the model graphs of $\mathcal Q$. As an example, taking $\mathcal P = \{ \chi \leq k\}$ and $\mathcal Q = \{ \chi \leq \ell\}$ for some integers $k < \ell$, in order to find an upper bound for the value of $c_{\{ \chi \leq k\}}(G)$ among graphs $G$ with $\chi(G) \leq \ell$, it is enough to find an upper bound on $c_{\{ \chi \leq k\}}(K_\ell)$, since $\{ \chi \leq \ell\}$ has the single model graph $K_\ell$.

Let us now move on to the definitions of the graph classes we study in this paper.

Let $k \in \Z$ be a parameter. The cover numbers by the graph classes $\{ \chi \leq k\}$ have already been studied, and their value is known exactly as per Theorem \ref{thm:hhm}. As discussed previously, the graph class $\{\chi \leq k\}$ has a single model graph, the complete graph $K_k$.


The \textit{Kneser graph} $KG(n,k)$ with parameters $n, k \in \N$, $n \geq 2k$, is the graph with vertex set $\binom{[n]}{k}$ and two vertices adjacent if and only if they are disjoint. The chromatic number of Kneser graphs is known from the classical result of Lovász \cite{Lovasz1978}, commonly referred to as the Lovász-Kneser theorem.

\begin{thm}[Lovász-Kneser \cite{Lovasz1978}] \label{thm:lovasz-kneser}
	For all $n \geq 2k$, we have \[
		\chi(KG(n,k)) = n - 2k + 2
	\]
\end{thm}

A \textit{fractional coloring} of a graph $G$ is a homomorphism from $G$ to a Kneser graph. The \textit{fractional chromatic number} of $G$ is the smallest value of $n/k$ with $G \to KG(n,k)$. This is well-defined for all finite graphs. For additional information, we refer to the textbook \cite{ScheinermanUllman2011book}. Occasionally, we will use the term \textit{$n/k$-colorable} to denote a graph with fractional chromatic number at most $n/k$.

Note that unlike the ordinary chromatic number, a graph $G$ having fractional chromatic number $n/k$ need not have a homomorphism to $KG(n,k)$. It might be the case that $G$ only admits a homomorphism to a larger Kneser graph; $KG(cn, ck)$ for some $c \in \N$. If $G$ only admits a homomorphism to a large Kneser graph, its ordinary chromatic number might still be arbitrarily large, as per Theorem \ref{thm:lovasz-kneser}.

Let $\beta \in \Q$ be a rational number parameter, with $\beta > 2$. We will study the graph classes $\{\chi_f \leq \beta\}$ in Section \ref{sec:fractional}. These graph classes, unlike their counterparts with the ordinary chromatic number, have an infinite set of model graphs; as each $KG(\beta k, k)$ with a suitable $k \in \Z$ is a model graph of $\{\chi_f \leq \beta\}$. Note that all of these model graphs have the same fractional chromatic number, but their ordinary chromatic numbers go to infinity.


The \textit{local chromatic number} of a graph $G$, denoted $\psi(G)$, is defined the following way. Take a proper coloring $\phi$ of $G$, and consider the local neighborhoods of each vertex. (We consider closed neighborhoods, i.e. the vertex itself is contained in its neighborhood.) If the maximum number of distinct colors appearing in the neighborhood of some vertex is $r$, then we say that the \textit{local value} of $\phi$ is $r$. The local chromatic number is then defined as the minimal local value across all proper colorings $\phi$.

We now define the graphs $U(m,r)$, which provide an alternative definition of the local chromatic number. Fix the natural numbers $m \geq r \geq 1$, and let \[
	V(U(m,r)) := \left\{ (\alpha, A) : \alpha \in [m], A \subseteq [m], \alpha \notin A, |A| \leq r-1  \right\}
\]
Two vertices $(\alpha, A), (\beta,B) \in V(U(m,r))$ are adjacent if and only if $\alpha \in B$ and $\beta \in A$.

It is clear that for any graph $G$, a local coloring of value $r$ exists if and only if there is a natural number $m$ such that $G \to U(m,r)$. In other words, for a constant $k \in \N$, the graph class $\{\psi \leq k\}$ is defined by the model graphs $\{ U(m,k) \mid m \in \N\}$. As with the fractional chromatic number case, we see that we have an infinite set of model graphs, whose local chromatic number is constant, but their ordinary chromatic number goes to infinity as per a result of Erdős et al. \cite{Erdos&al1986}.


\begin{thm}[Erdős et al. {\cite[Theorem~2.3]{Erdos&al1986}}] \label{thm:erdos-umr-chromat} 
	Fix $r \in \N$ with $r \geq 2$. Then \[
		\lim_{m \to \infty} \chi(U(m,r)) = \infty
	\]
\end{thm}

It is to be noted that the fractional chromatic number of the graphs $U(m,r)$ is known, as per the following result of Körner, Pilotto and Simonyi \cite{Korner&al2005}.

\begin{thm}[Körner, Pilotto, Simonyi \cite{Korner&al2005}]
	For all values of $m, r \in \N$ with $m \geq r \geq 2$, we have\[
		\chi_f(U(m,r)) = r
	\]
\end{thm}

This implies that for all graphs $G$, we have $\chi_f(G) \leq \psi(G)$. Combining this with other well-known and easy-to-verify inequalities, the following can be stated about the relationship of the graph parameters considered in this paper.

\begin{thm} \label{thm:graphparameter-chain}
	For all graphs $G$, we have \[
		\omega(G) \leq \chi_f(G) \leq \psi(G) \leq \chi(G)
	\]
\end{thm}

Obviously, bipartite graphs satisfy $\omega(G) = \chi(G) \leq 2$, and so the following theorem is trivially true.

\begin{thm} \label{thm:bip2-equivalence}
	Let $G$ be a graph. The following are equivalent. \begin{enumerate}
		\item $G$ is bipartite with at least one edge
		\item $\chi(G) = 2$
		\item $\chi_f(G) = 2$
		\item $\psi(G) = 2$
	\end{enumerate}
\end{thm}

We study cover numbers of the form $c_{\{* \leq \beta\}}$, where $\beta$ is a constant, and $*$ is one of the graph parameters $\chi$, $\chi_f$ or $\psi$. As per Theorem \ref{thm:bip2-equivalence}, we can say that the case of $\beta = 2$ is equivalent for all three choices of $*$, and an exact formula for this cover number is already given by Theorem \ref{thm:hhm}. Therefore, the cases where $\beta > 2$ are the most interesting.




\section{Analogous questions with the fractional chromatic number} \label{sec:fractional}


In this section, we consider the graph classes $\{\chi_f \leq \beta\}$ with a rational constant $\beta$. In this section, we will always assume $\beta > 2$. Note that Theorem \ref{thm:bip2-equivalence} implies that the case of $\beta = 2$ is not interesting. We ask a question analogous to Theorem \ref{thm:hhm}, replacing the chromatic number with the fractional chromatic number, and the constant $k \in \N$ with a constant $\beta \in \Q$.

\begin{question} \label{q:fractional-analog}
	Is it true that for all graphs $G$ and rational numbers $\beta > 2$, the following holds? \[
		c_{\{\chi_f \leq \beta\}}(G) = \left\lceil \frac{\log \chi_f(G)}{\log \beta} \right\rceil
	\]
\end{question}

Counterintuitively, this is false. Consider the Kneser graph $KG(12,2)$, and let $\beta = 5/2$. Based on the formula, since $\left\lceil \frac{\log 6}{\log 2.5} \right\rceil = \left\lceil 1.955 \right\rceil = 2$, we would expect that $KG(12,2)$ is coverable by two $5/2$-colorable graphs. But since $KG(12,2)$ has a clique of size six, the ``little Ramsey" theorem (i.e. that $R(3,3) = 6$) tells us that any covering of it by two graphs will necessarily have a monochromatic triangle, meaning that it is impossible to cover $KG(12,2)$ by two $5/2$-colorable graphs.

Due to the direct analog of Theorem \ref{thm:hhm} being false, we will prove the following weaker statement.

\begin{thm} \label{thm:frac-upper-lower}
	Let $G$ be any graph, and $\beta > 2$ a rational number. Then \[
		\left\lceil \frac{\log \chi_f(G)}{\log \beta} \right\rceil \leq c_{\{\chi_f \leq \beta\}}(G) \leq c_{\{ \chi = \chi_f \}}(G) \cdot \left\lceil \frac{\log\chi_f(G)}{\log \left\lfloor \beta \right\rfloor} \right\rceil
	\]
\end{thm}


In this weakening, the lower bound on the cover number is exactly as we would expect. The upper bound, on the other hand, involves the cover number by $\{ \chi_f = \chi \}$. This means that for graphs satisfying $\chi_f(G) = \chi(G)$, such as perfect graphs and a variety of other graph families, we have an asymptotically sharp bound.

Since $\{ \chi_f = \chi \} \supseteq \{ \chi = \omega\}$, we can use Theorem \ref{thm:coverno-bijection} and Theorem \ref{thm:m-omega} to arrive at the following conclusion, proving Theorem \ref{thm:INTRO_frac-upper-lower}.

\begin{cor} \label{cor:frac-upper-lower}
	Let $G$ be any graph, and $\beta > 2$ a rational number. Then \[
		\left\lceil \frac{\log \chi_f(G)}{\log \beta} \right\rceil \leq c_{\{\chi_f \leq \beta\}}(G) \leq \left\lceil \frac{\log \chi(G)}{\log \omega(G)} \right\rceil \cdot \left\lceil \frac{\log\chi_f(G)}{\log \left\lfloor \beta \right\rfloor}  \right\rceil
	\]
\end{cor}



We prove Theorem \ref{thm:frac-upper-lower} by proving the upper and lower bounds separately. First, the lower bound.

\begin{thm}
	Let $G$ be any graph, and $\beta > 2$ a rational number. Then \[
		c_{\{\chi_f \leq \beta\}}(G) \geq \left\lceil \frac{\log \chi_f(G)}{\log \beta} \right\rceil
	\]
\end{thm}

\begin{proof}
	Take an optimal covering of $G$ by $c = c_{\{\chi_f \leq \beta\}}(G)$-many graphs with fractional chromatic number at most $\beta$. Call these graphs $G_1, \dots, G_c$. Suppose that $G_i \to KG(p_i, q_i)$ for some $p_i, q_i \in \N$ with $\frac{p_i}{q_i} = \beta$. Then for $p = \max\limits_{i} p_i$ and $q = \max\limits_{i} q_i$, it is easy to see that $p/q = \beta$, and $G_i \to KG(p,q)$ for all $i$.
	
	We will construct a product coloring the following way. For each vertex $v \in V(G)$, let $\phi_i(v)$ denote the vertex of $KG(p,q)$ assigned to $v$ in the $\frac{p}{q}$-coloring of $G_i$. We can naturally think of the vertices of $KG(p,q)$ as a list of $q$ different colors chosen from $\{1, \dots, p\}$. Let $\tau$ be the mapping that takes a vertex $v \in V(G)$ to the set \[
		\tau(v) := \left\{ (\phi_1, \dots, \phi_c) \mid \phi_1 \in \phi_1(v), \dots, \phi_c \in \phi_c(v) \right\}
	\]
	
	Then, for each vertex $v \in V(G)$, $\tau(v)$ is a set containing $q^c$ different lists of $c$ colors each, chosen from a total of $p^c$ possible lists. In other words, $\tau(v)$ can be naturally understood as a vertex of $KG(p^c, q^c)$. It is now easy to see that $\tau$ is a valid $\frac{p^c}{q^c}$-coloring; a given list $(\phi_1, \dots, \phi_c)$ cannot appear in both $\tau(v)$ and $\tau(w)$ for two adjacent vertices $v$ and $w$, as at least one $G_i$ must cover the edge between them, and thus each list in $\tau(v)$ must have a different $\phi_i$ from the lists in $\tau(w)$, since $\phi_i(v)$ is disjoint from $\phi_i(w)$. 
	
	Thus, $\chi_f(G) \leq \left(p/q\right)^c = \beta^c$, which can be rearranged into the desired inequality.
\end{proof}


To prove the upper bound in Theorem \ref{thm:frac-upper-lower}, we prove the following assertion for graphs in $\{\chi_f = \chi\}$. Incidentally, combined with the lower bound, this leads to an asymptotically sharp result on this restricted family of graphs.

\begin{thm}
	Let $\beta$ be a rational number with $\beta > 2$. If $G \in \{\chi_f = \chi\}$, then \[
		\frac{\log\chi_f(G)}{\log \left\lfloor \beta \right\rfloor} \geq c_{\{ \chi_f \leq \beta \}}(G)
	\]
\end{thm}

\begin{proof}
	Let $G$ be a graph as in the theorem, and indirectly assume that \[
		\frac{\log\chi_f(G)}{\log \left\lfloor \beta \right\rfloor} < c_{\{ \chi_f \leq \beta \}}(G)
	\] or equivalently \[
		\chi_f(G) < \left\lfloor \beta \right\rfloor^{c_{\{ \chi_f \leq \beta \}}(G)}
	\]
	
	Notice that since $\{ \chi_f \leq \beta \} \supseteq \{ \chi_f \leq \left\lfloor \beta \right\rfloor \} \supseteq \{ \chi \leq \left\lfloor \beta \right\rfloor \}$, we can use Theorem \ref{thm:coverno-bijection} to get \[
		\chi_f(G) < \left\lfloor \beta \right\rfloor^{c_{\{ \chi_f \leq \beta \}}(G)} \leq \left\lfloor \beta \right\rfloor^{c_{\{ \chi \leq \left\lfloor \beta \right\rfloor \}}(G)}
	\] and from here, since $\lfloor \beta \rfloor$ is an integer, we can use Theorem \ref{thm:hhm} to conclude that the right side of this is equal to $\chi(G)$. We arrive at $\chi_f(G) < \chi(G)$, a contradiction.
\end{proof}



Generalizing the previous result to all graphs, we arrive at the following conclusion.

\begin{thm}
	For all graphs $G$ and rational numbers $\beta > 2$, we have \[
		c_{\{ \chi_f \leq \beta \}}(G) \leq c_{\{ \chi = \chi_f \}}(G) \cdot \frac{\log\chi_f(G)}{\log \left\lfloor \beta \right\rfloor}
	\]
\end{thm}

\begin{proof}
	Take an optimal covering $G = G_1 \cup \dots \cup G_c$ by $\{ \chi = \chi_f \}$-graphs, where $c = c_{\{ \chi = \chi_f \}}$. Each of the covering graphs can be covered by $\frac{\log\chi_f(G)}{\log \left\lfloor \beta \right\rfloor}$-many graphs as per the previous theorem, so taking them all together gives a covering of $G$ by the desired number of $\beta$-colorable graphs.
\end{proof}


This concludes the proof of Theorem \ref{thm:INTRO_frac-upper-lower}.

\section{Small cases with the fractional chromatic number} \label{sec:specificcases}




Let us now investigate the cover numbers of the form $c_{\{ \chi_f \leq \beta\}}$ of certain small Kneser graphs. In particular, we return to the disproof of Question \ref{q:fractional-analog}. We have noted previously that $KG(12,2)$ cannot be covered by two $2.5$-colorable graphs. 
It is now a natural question to ask about the Kneser graphs $KG(n,2)$, how much do we have to decrease the value of $n$ to arrive at a graph that is coverable by two $2.5$-colorable 
graphs. 

Actually, the disproof of Question \ref{q:fractional-analog} gives an even stronger result, namely, that $KG(12,2)$ cannot be covered by two triangle-free graphs. This triangle-free case, i.e. $c_{\{\omega \leq 2\}}$, is essentially the inverse of a Ramsey-theoretic question in Kneser graphs, which has been recently resolved by computer in Heath et al. \cite{Heath&al2025}.

\begin{thm}[Heath et al. \cite{Heath&al2025}]
	The following two statements hold.
    \begin{enumerate}
    	\item The Kneser graph $KG(9,2)$ cannot be covered by two triangle-free graphs.
    	\item The Kneser graph $KG(8,2)$ can be covered by two triangle-free graphs.
    \end{enumerate}
\end{thm}

For $KG(9,2)$, this means that a covering by two $2.5$-colorable graphs is impossible. For $KG(8,2)$, they find a covering by two triangle-free graphs. One of the covering graphs they find has fractional chromatic number exceeding three, however. Another triangle-free covering of $KG(8,2)$ was found independently in \cite{Yang2025personal}, this covering also uses a graph with high fractional chromatic number. Both constructions were found by computer. Therefore, the question of whether $KG(8,2)$ can be covered by two $2.5$-colorable graphs is still open.

For $KG(7,2)$, we present a construction of a covering by two $2.5$-colorable graphs. In fact, we can state the construction generally for all Kneser graphs of the form $KG(3k+1,k)$ for some $k$.

\begin{thm} \label{thm:construction}
	Let $k \geq 1$ be an integer, then $KG(3k+1,k)$ can be covered by two triangle-free graphs, with fractional chromatic numbers $\frac{3k-1}{k}$ and $2$.

\end{thm}

\begin{proof}
	
	Let $G = KG(3k+1,k)$ for some $k$, and let $H$ be the subgraph induced by $\{ v \in V(G) \mid 3k \notin v \wedge 3k+1 \notin v\}$. This is easily seen to be isomorphic to the smaller Kneser graph $KG(3k - 1, k)$. Let $A = \{ v \in V(G) \mid 3k \in v\}$ and $B = \{ v \in V(G) \mid 3k + 1 \in v\} \setminus A$. Since every element of $A$ (resp. $B$) contains $3k$ (resp. $3k + 1$), these sets are both independent in $G$.

	
	Construct the covering in the following way. Let the first covering graph contain the edges inside $H$, as well as the edges connecting $A$ and $B$. This has fractional chromatic number $\frac{3k-1}{k}$ because of $H$. Let the second covering graph contain the edges between $H$ and $A \cup B$. Clearly, this is a bipartite graph between $H$ and $A \cup B$.
\end{proof}

\begin{figure} \label{fig:construction}
	\centering
	\captionsetup{justification=centering}
	\begin{subfigure}{0.35\textwidth}
		\resizebox{\linewidth}{!}{
		\begin{tikzpicture}
			\input{figure1.tex}
		\end{tikzpicture}}
		\caption{The construction \\ of Theorem \ref{thm:construction}}
	\end{subfigure}
	\qquad
	\begin{subfigure}{0.55\textwidth}
		\resizebox{\linewidth}{!}{
		\begin{tikzpicture}
			\input{figure2.tex}
		\end{tikzpicture}}
		\caption{The construction \\ of Theorem \ref{thm:construction2}}
	\end{subfigure}
	
\end{figure}


From here, we can consider slightly larger small cases as well. For the case of covering $KG(n,3)$ by two triangle-free graphs, a recent result by Heath et al. proved that $KG(13,3)$ cannot be covered by two triangle-free graphs. Taking the construction from Theorem \ref{thm:construction} with $k = 3$ shows that $KG(10,3)$ however can be covered in this way. Therefore, only the cases of $KG(11,3)$ and $KG(12,3)$ remain unresolved.

Instead of asking about coverings by two graphs, we can also ask what is the smallest number $k$ such that $KG(k,2)$ can be covered by three $2.5$-colorable graphs. In Theorem \ref{thm:construction2}, we provide a construction for covering $KG(3k + 6, k)$ by three graphs.


For the case of three covering graphs, the construction in Theorem \ref{thm:construction} can be improved in the following way. We do not know of an asymptotic improvement.


\begin{thm} \label{thm:construction2}
	Let $k > 1$ be an integer, then $KG(3k+6,k)$ can be covered by three graphs, all triangle-free, with fractional chromatic numbers $\frac{3k-1}{k}$, $2.5$ and $2$.

\end{thm}

\begin{proof}
	
	We use a similar construction as in Theorem \ref{thm:construction}. Let $G = KG(3k+6, k)$ for some $k$, and let $H = \{ v \in V(G) \mid \{3k, 3k + 1, \dots, 3k + 6\} \cap v = \emptyset\}$. This is once again easily seen to be isomorphic to the smaller Kneser graph $KG(3k - 1, k)$. Let $B_1 := \{ v \in V(G) \mid 3k \in v\}$ and $B_2 := \{ v \in V(G) \mid 3k+1 \in v\} \setminus B_1$. Each vertex of $B_1$ (resp. $B_2$) contains the element $3k$ (resp. $3k + 1$), therefore, these are independent sets in $G$. For $i = 0, \dots, 4$, let $A_i = \{v \in V(G) \mid 3k+2+i \in v\} \setminus \left(\cup_{j = 0}^{i-1} A_j \cup B_1 \cup B_2 \right)$. Since every vertex of $A_i$ contains $3k + 2 + i$, each $A_i$ is also independent.

		
	Construct the covering of $G$ in the following way. Let the first covering graph $G_1$ contain the edges inside $H$, the edges connecting $B_1$ and $B_2$, as well as the edges connecting $A_i$ and $A_{i+1}$ for each $i \in \{0, \dots, 4\}$, where the addition is understood modulo $5$. Let $G_2$ contain the edges between $A_i$ and $A_{i+2}$, where the addition is understood modulo $5$, and the edges connecting $H$ and $B_1 \cup B_2$. Finally, let $G_3$ contain the edges between $H \cup B_1 \cup B_2$ and $\cup_{i = 0}^4 A_i$. It is easy to see that none of the three graphs contain triangles.

	
	In this construction, $G_1$ will have fractional chromatic number $\frac{3k-1}{k}$, $G_2$ will have $2.5$ due to containing a copy of $C_5$, and $G_3$ will be bipartite.
\end{proof}

The constructions in Theorems \ref{thm:construction} and \ref{thm:construction2} can be generalized to more colors, by adding more independent sets. Specifically, when we consider the case of $c$ colors, we can use the construction for $c-1$ colors, and add another $R_{c-1}(3) - 1$-many extra independent sets, where $R_{c-1}(3)$ is the $(c-1)$-colored Ramsey number for a monochromatic triangle, and use the colors $1, \dots, c-1$ to color the edges connecting these new independent sets, in a way similar to the $A_i$-s in Theorem \ref{thm:construction2}. Then we can use the extra color for the bipartite graph connecting the old vertices with the new ones. The details of this generalization are left to the reader.

A small improvement can be made on the construction of Theorem \ref{thm:construction2} if we assume $k \equiv 1$ modulo $5$. In this case, $KG(3k + 7, k)$ is also coverable by three triangle-free graphs, in the following way. Let $k = 5\ell +1$ and let the vertex sets $A_1, \dots, A_5 \subseteq V(KG(3k + 7, k))$ be defined by \[A_i := \Bigl\{ v \Bigm| |v \cap \{(3\ell + 2)(i-1) + 1, \dots, (3\ell + 2)i  \}| \geq \ell + 1 \Bigr\} \setminus \bigcup_{j=1}^{i-1} A_j\] Every vertex of $KG(15\ell + 10, 5\ell + 1)$ will belong to exactly one of these sets, as per the pigeonhole principle, and they are easily seen to be homomorphic to $KG(3\ell + 2, \ell + 1)$, implying that they are all triangle-free. Thus, let the first covering graph, $G_1$, consist of the edges inside the $A_i$-s, let $G_2$ consist of the edges connecting $A_i$ with $A_{i+1}$ (with the addition understood modulo $5$), and let $G_3$ consist of the edges connecting $A_i$ with $A_{i+2}$. It is left to the reader to verify that these graphs have fractional chromatic numbers at most $\frac{3\ell + 2}{\ell + 1}$, $2.5$ and $2.5$ respectively.

Importantly, we do not know of an asymptotic improvement, even if we allow a large amount of colors. The following question therefore arises naturally.

\begin{question}
	Let $\eps > 0$ be a real number, and $c \geq 3$ an integer. Is there a covering of the edges of $KG\left( (3+\eps)k, k \right)$ with $c$ triangle-free graphs, for all $k$?
\end{question}

This question, along with others of a similar nature, will be considered in Section \ref{sec:eightconj}, where it is equivalent to the converse of Question \ref{que:eight}, part (1).



\section{Analogous questions with the local chromatic number} \label{sec:local}

In this section, we consider the graph classes $\{\psi \leq k\}$ for $k$ a positive integer. In this case, we find an upper bound on $c_{\{\psi \leq k\}}$ in terms of $\psi$, and a lower bound on $c_{\{\psi \leq k\}}$ in terms of $\omega$.


To prove the upper bound of Theorem \ref{thm:INTRO_local-upper}, we first consider the simplest non-trivial case, where $k = 3$. In addition, we restrict our attention to the case where $G$ is one of the graphs $U(m,r)$. Due to Corollary \ref{cor:upper-bound-reasoning}, finding an upper bound on $c_{\{\psi \leq k\}}(U(m,r))$ for all $m$ immediately implies that the same upper bound holds for all graphs $G$ with $\psi(G) = r$.

\begin{thm} \label{thm:Umr-covering}
	For each $r$, there is a decomposition of the edges of $U(m,r)$ into $(r-1)^2$-many graphs of local chromatic number at most three.
\end{thm}

\begin{proof}
	Let $r$ be fixed. For each $i, j \in \{0, \dots, r - 2 \}$, let $E_{i,j}$ be defined in the following way.
	
	\begin{equation*} 
		E_{i,j} := \left\{ (\alpha, A) \sim (\beta, B) : \begin{aligned}
				& \alpha \in B, \beta \in A \\ & |\{ a \in A | a < \beta\}| = i \\ & |\{ b \in B | b > \alpha\}| = j  \\
				& \alpha < \beta 
		\end{aligned} \right\}
	\end{equation*}
	
	In other words, each edge of $U(m,r)$ is put into one of the $E_{i,j}$-s based on how many elements of $A$ are less than $\beta$, and how many elements of $B$ are more than $\alpha$. Since $\beta \in A$ and $\alpha \in B$, the cases of $i = r - 1$ or $j = r - 1$ cannot occur. The total number of covering graphs is exactly $(r-1)^2$, the number of possible $(i,j)$ pairs.
	
	We claim that these $E_{i,j}$-s cover all edges of $U(m,r)$, and that each $E_{i,j}$ has local chromatic number at most three, with the local coloring of value $3$ given by $(\alpha, A) \mapsto \alpha$. The first assertion is clear, each edge falls into exactly one of the $E_{i,j}$-s, depending on the properties of the endpoints. Now let us fix some pair $(i,j)$ and consider the graph induced by $E_{i,j}$. 
	
	We claim that the neighborhood of each vertex $(\gamma, C)$ is colored with at most two colors. Indeed, suppose $(\alpha, A)$ is a neighbor of $(\gamma, C)$ in $E_{i,j}$, with $\alpha < \gamma$. Then $\alpha$ must be exactly the $(|C| - j)$-th element of $C$ by the definition of $E_{i,j}$. If instead $(\beta, B)$ is a neighbor of $(\gamma, C)$ with $\gamma < \beta$, then $\beta$ can only be the $(i + 1)$-st element of $C$. Therefore, only these two elements of $C$ may be the colors of a neighbor of $(\gamma, C)$ in $E_{i,j}$. Therefore, each $E_{i,j}$ has local chromatic number at most three, and the theorem is proved.
	
\end{proof}

Now let $G$ be any graph with $\psi(G) = r$. We know that $G \to U(m, r)$ for some $m$. This implies that $c_{\{ \psi \leq 3\}}(G) \leq c_{\{ \psi \leq 3\}}(U(m,r))$, by Theorem \ref{thm:graph-class-defd-by-homomorphism}. Therefore, the following corollary holds.

\begin{cor} \label{cor:local-3}
	For all graphs $G$, we have \[
		c_{\{\psi \leq 3\}}(G) \leq (\psi(G) - 1)^2
	\]
\end{cor}


We can generalize the above arguments from $\{\psi \leq 3\}$ to all graph families of the form $\{\psi \leq k\}$ for some integer $k \geq 3$. The following theorem, mentioned in the introduction as Theorem \ref{thm:INTRO_local-upper}, can be proved. It is clear that this is a direct generalization of Corollary \ref{cor:local-3}.

\begin{thm} \label{thm:local-upper-general}
	For all graphs $G$ and integers $k \geq 3$, we have \[
		c_{\{\psi \leq k\}}(G) \leq \left\lceil \frac{\psi(G) - 1}{\left\lfloor \frac{k-1}{2} \right \rfloor} \right\rceil \cdot \left\lceil \frac{\psi(G) - 1}{\left\lceil \frac{k-1}{2} \right \rceil} \right\rceil
	\]
\end{thm}

\begin{proof}
	As we have seen previously, it is sufficient to prove this in the case of $G = U(m,r)$ for some $m, r$, where $r = \psi(G)$. We will let $E_{i,j}$ for $i \in \{0, \dots, r - 2\}$ and $j \in \{0, \dots, r-2\}$ be defined as in the proof of Theorem \ref{thm:Umr-covering}.

	Define $a := \left\lfloor \frac{k-1}{2} \right \rfloor$ and $b := \left\lceil \frac{k-1}{2} \right \rceil$. Note that $a + b + 1 = k$ always holds for these numbers. Define the sets $I = \{0, \dots, r - 2\}$ and $J = \{0, \dots, r - 2\}$. These are exactly the possible values of $i$ and $j$ from the proof of Theorem \ref{thm:Umr-covering}.
	
	Divide $I$ into $a$-element subsets as $I = I_1 \cup \dots \cup I_{\left\lceil \frac{r-1}{a} \right\rceil}$ and $J$ into $b$-element subsets as $J = J_1 \cup \dots \cup J_{\left\lceil \frac{r-1}{b} \right\rceil}$. Now define the edge sets $\hat{E}_{i,j}$ as follows.
	
	\begin{equation*} 
		\hat{E}_{i,j} := \bigcup_{i^* \in I_i, j^* \in J_j} E_{i^*, j^*}
	\end{equation*}
	
	We claim that each $\hat{E}_{i,j}$ has local chromatic number at most $k$, and that the local coloring of value $k$ is given by $(\alpha, A) \mapsto \alpha$. Indeed, fix a vertex $(\gamma, C) \in V(U(m,r))$, 
	and suppose that $(\alpha, A)$ is a neighbor of $(\gamma, C)$ in $\hat{E}_{i,j}$ with $\alpha < \gamma$. This restricts the possible values of $\alpha$, specifically, $\alpha$ can only be the $(|C| - j^*)$-th element of $C$ for some $j^* \in J_j$. This is a total of $b$-many potential colors for a neighbor of $(\gamma, C)$. Alternatively, if $(\beta, B)$ is a neighbor of $(\gamma, C)$ in $\hat{E}_{i,j}$ with $\gamma < \beta$, then similarly $\beta$ can only be the $(i^* + 1)$-st element for some $i^* \in I_i$, which allows for a total of $a$-many potential colors for each neighbor. In total, the neighbors of $(\gamma, C)$ may have at most $a + b = k - 1$ distinct colors, implying that $\hat{E}_{i,j}$ has local chromatic number at most $k$.
	
	The $\hat{E}_{i,j}$-s clearly cover $G$, and they form covering of $G$ by $\left\lceil \frac{r - 1}{a} \right\rceil \cdot \left\lceil \frac{r - 1}{b} \right\rceil$-many graphs in $\{\psi \leq k\}$, proving the theorem.
\end{proof}

Let us move on to the proof of the lower bound on $c_{\{\psi \leq k\}}$ in Theorem \ref{thm:INTRO_local-upper}. This lower bound is given in terms of a function $f$, which is defined recursively.

\begin{thm} \label{thm:local-lower}
	Let $k \geq 3$ be an integer, and $G$ be any graph. The \[
		c_{\{\psi \leq k\}}(G) \geq f_k^{-1}(\omega(G))
	\] where $f_k$ denotes the function defined recursively by $f_k(1) = k$ and $f_k(c) = c (k-1) f_k(c-1)$ for $c \geq 2$.
\end{thm}

\begin{proof}
	First, notice that $c_{\{\psi \leq k\}}(G) \geq c_{\{\psi \leq k\}}(K_{\omega(G)})$ for any graph $G$, since there always exists a homomorphism $K_{\omega(G)} \to G$. Therefore, let $r := \omega(G)$ and $c := c_{\{\psi \leq k\}}(K_r)$. We will prove that a function $f_k$ will satisfy $f_k(c) \geq r$, which will prove the theorem. In other words, we need to prove that if $K_r$ is covered by $c$-many graphs from $\{\psi \leq k\}$, then it may only have $f_k(c)$-many vertices. 
	
	We proceed by induction on $c$, starting with the base case, $c = 1$. This case is trivial, as a graph with $c_{\{\psi \leq k\}}(G) = 1$ must have $\psi(G) \leq k$, which means that $r \leq k$ must hold.
	
	
	Now suppose inductively that we already know the theorem for $c-1$. Once again, suppose that $K_r$ is covered by $c$ graphs $G_1, G_2, \dots, G_c$, each of which has local chromatic number at most $k$. Let $v \in V(K_r)$ be any vertex. Let $A_i := \{ w \in V(K_r) \mid vw \in E(G_i)\}$ for $i = 1, 2, \dots, c$. Since $G_1$ has local chromatic number at most $k$, the vertices in $A_1$ can have one of $k-1$ possible colors in $G_1$. Denote these color sets by $A_1^{(1)}, \dots, A_1^{(k-1)}$, and define the sets $A_i^{(j)}$ analogously for $i \in \{1, \dots, c\}$ and $j \in \{1, \dots, k-1\}$. The vertices in $A_1^{(1)}$ are all colored with the same color in $G_1$, hence they are a clique in $G_2 \cup \dots \cup G_c$. Thus, $(G_2 \cup \dots \cup G_c) \cap A_1^{(1)}$ is a complete graph covered by $(c-1)$-many graphs of local chromatic number at most $k$, consequently $A_1^{(1)}$ has at most $f_k(c-1)$ vertices. Summing this up for all $A_i^{(j)}$'s, we therefore have \[
		r \leq c \cdot (k-1) \cdot f_k(c-1) = f_k(c)
	\] which proves the theorem.
\end{proof}

Solving the recursion defining the functions $f_k$, one can conclude that $f_k(x)$ must be of the order $O(x! \cdot k^x)$, therefore, the proof of Theorem \ref{thm:INTRO_local-upper} is concluded.

The proof of Theorem \ref{thm:local-lower} illustrates very well the relationship between Ramsey theory and the study of cover numbers of this type. Indeed, the statement of Theorem \ref{thm:local-lower} can be restated in the following fashion.

\begin{cor}
	For any integer $k \geq 2$, for a large enough $r$, if the edges of the complete graph $K_r$ are colored with $c$ colors, then there is a monochromatic subgraph with local chromatic number at least $k + 1$.
\end{cor}

The proof of this lower bound has a very similar structure to the proofs of certain Ramsey-theoretic upper bounds, which makes sense, as the cover numbers we study here are, in some sense, inverses of analogous Ramsey numbers. However, it is to be noted that this direct proof gives a better bound than a reduction to Ramsey theory.

Indeed, one could also give a lower bound on $c_{\{\psi \leq k\}}$ by giving a lower bound on $c_{\{\omega \leq k\}}$ instead, which is a lower bound on $c_{\{\psi \leq k\}}$ by Theorem \ref{thm:coverno-bijection}. It is easy to see that the cover number $c_{\{\omega \leq k\}}$ is the inverse of the multicolor Ramsey number $R_c(k+1)$ with respect to the number of colors, $c$. However, even the recently improved, best upper bound from \cite{Balister&al2026} gives a higher estimate for $R_c(k+1)$ than the function $f_k(c)$ from the proof. Consequently, the proof presented here gives a stronger lower bound.

It is an open question whether it it possible to give a lower bound on $c_{\{\psi \leq k\}}$ that goes to infinity depending on $\psi(G)$. As it stands, it is possible that there exist graphs with high local chromatic number, that can nevertheless be covered by a constant number of $\{\psi \leq k\}$-graphs with a constant $k$.


\section{Some asymptotic questions} \label{sec:eightconj}

In Theorem \ref{thm:frac-upper-lower}, we have found an upper and a lower bound on $c_{\{\chi_f \leq \beta\}}$ in terms of some graph parameters and $\beta$. Suppose we have a fixed $\beta$. In that case, the lower bound depends only on the fractional chromatic number, but the upper bound also includes $\chi$, the ordinary chromatic number. Since $\chi$ can be arbitrarily large among graphs with constant fractional chromatic number, due to Theorem \ref{thm:lovasz-kneser}, there are families of graphs where the lower bound is constant, but the upper bound can be arbitrarily large.

Two such graph families have already been discussed. The family of Kneser graphs $KG(\alpha k, k)$ with fixed $\alpha > \beta$ have constant fractional chromatic number $\alpha$, but their chromatic number goes to infinity as $k \to \infty$. Similarly, the graphs $U(m,\alpha)$ discussed in Section \ref{sec:local} form a sequence of graphs with fractional chromatic number $\alpha$ and chromatic number going to infinity as $m \to \infty$. In this section, we will investigate the asymptotic behavior of certain cover numbers of these two graph families. In addition to the motivating example of $c_{\{\chi_f \leq \beta\}}$, we will examine analogous questions for $c_{\{\chi \leq \beta\}}$, $c_{\{\psi \leq \beta\}}$ and $c_{\{\omega \leq \beta\}}$.

The following eight questions can be asked regarding the asymptotic behavior of the cover numbers of the aforementioned families of graphs.

\begin{minipage}{\textwidth} 
\begin{question} \label{que:eight}
	Let $\alpha > \beta > 2$ be rational numbers.
	\begin{figure}[H]
	  \centering
	  \begin{subfigure}[b]{.4\linewidth}
	    \begin{equation}
	      \lim_{k \to \infty} c_{\{\omega \leq \beta\}}(KG(\alpha k, k)) \stackrel{?}{=} \infty
	    \end{equation}
	    \label{que:kg-by-omega}
	  \end{subfigure}
	  \begin{subfigure}[b]{.4\linewidth}
	    \begin{equation}
	      \lim_{m \to \infty} c_{\{\omega \leq \beta\}}(U(m, \alpha)) \stackrel{?}{=} \infty
	    \end{equation}
	    \label{que:umr-by-omega}
	  \end{subfigure}
	  \begin{subfigure}[b]{.4\linewidth}
	  	    \begin{equation}
	  	      \lim_{k \to \infty} c_{\{\chi_f \leq \beta\}}(KG(\alpha k, k)) \stackrel{?}{=} \infty
	  	    \end{equation}
	  	    \label{que:kg-by-chif}
	  \end{subfigure}
	  \begin{subfigure}[b]{.4\linewidth}
	  	    \begin{equation}
	  	      \lim_{m \to \infty} c_{\{\chi_f \leq \beta\}}(U(m, \alpha)) \stackrel{?}{=} \infty
	  	    \end{equation}
	  	    \label{que:umr-by-chif}
	  \end{subfigure}
	  \begin{subfigure}[b]{.4\linewidth}
	  	    \begin{equation}
	  	      \lim_{k \to \infty} c_{\{\psi \leq \beta\}}(KG(\alpha k, k)) \stackrel{?}{=} \infty
	  	    \end{equation}
	  	    \label{que:kg-by-psi}
	  \end{subfigure}
	  \begin{subfigure}[b]{.4\linewidth}
	  	    \begin{equation}
	  	      \lim_{m \to \infty} c_{\{\psi \leq \beta\}}(U(m, \alpha)) \stackrel{?}{=} \infty
	  	    \end{equation}
	  	    \label{que:umr-by-psi}
	  \end{subfigure}
	  \begin{subfigure}[b]{.4\linewidth}
	  	    \begin{equation}
	  	      \lim_{k \to \infty} c_{\{\chi \leq \beta\}}(KG(\alpha k, k)) \stackrel{?}{=} \infty
	  	    \end{equation}
	  	    \label{que:kg-by-chi}
	  \end{subfigure}
	  \begin{subfigure}[b]{.4\linewidth}
	  	    \begin{equation}
	  	      \lim_{m \to \infty} c_{\{\chi \leq \beta\}}(U(m, \alpha)) \stackrel{?}{=} \infty
	  	    \end{equation}
	  	    \label{que:umr-by-chi}
	  \end{subfigure}
	 \label{fig:8questions}
	\end{figure}
\end{question}
\end{minipage}

Note the arrangement of the questions. Each question implies all those below it, because of the fact that $\omega(G) \leq \chi_f(G) \leq \psi(G) \leq \chi(G)$ for all graphs, and Theorem \ref{thm:coverno-bijection}. Each of the even-numbered questions further implies the odd-numbered one next to it by Theorem \ref{thm:graph-class-defd-by-homomorphism}, since for all $m$, $U(m,\alpha) \to KG(\alpha k, k)$ for some $k$ dependent on $m$.

The answers to Questions (7) and (8) are readily answered by the Harary-Hsu-Miller theorem (Theorem \ref{thm:hhm}); both of these statements are clearly true, as the chromatic numbers of $KG(\alpha k, k)$ and $U(m, \alpha)$ both tend to infinity, as per Theorems \ref{thm:lovasz-kneser} and \ref{thm:erdos-umr-chromat} respectively. Also, due to Theorem \ref{thm:local-upper-general}, we can see that the answer to Question (6) is negative. This further implies that (4) and (2) must also be false.

Questions (1), (3) and (5) remain unresolved. Question (1) can be rephrased as a Ramsey-theoretic question of independent interest:

\begin{question}
	Let $c \geq 2$ and $n \geq 2$ be integers and $\alpha > n$ be a rational number. Is it true that for large enough $k$, all colorings of the edges of $KG(\alpha k, k)$ with $c$ colors will necessarily contain a monochromatic $K_n$?
\end{question}

\section{Open problems}





The following questions remain regarding the general theory.

\begin{enumerate}
	\item Resolve the remaining three questions out of the eight in Question \ref{que:eight}.
	\item Find a sequence of graphs attaining the upper bound of Theorem \ref{thm:frac-upper-lower}. (Such a sequence only exists if the answer to Question \ref{que:eight}, part (3) is positive.)
	\item For each integer $k \geq 3$, give a general lower bound for $c_{\{\psi \leq k\}}$ in terms of $\psi$. 
\end{enumerate}

Regarding small, specific cases of Kneser graph coverings, the following open questions remain.

\begin{enumerate}
	\item Is it possible to cover $KG(8,2)$ by two $2.5$-colorable graphs?
	\item What is the smallest $\beta \in \Q$ such that $KG(7,2)$ (resp. $KG(8,2)$) can be covered by two $\beta$-fractional-colorable graphs? We know that the optimal $\beta$ satisfies $2 < \beta \leq 3$ in both cases.
	\item Is it possible to cover $KG(11,3)$ (resp. $KG(12,3)$) by two $2.5$-colorable graphs? (This question has already been asked for triangle-free coverings in \cite{Heath&al2025})
\end{enumerate}

\section{Acknowledgements}

Anna Gujgiczer was supported by the National Research, Development and Innovation Office NKFIH (Excellence program, Grant Nr. 153829). Kenta Ozeki was supported by JSPS KAKENHI Grant Numbers 22K19773, 23K03195 and 26K00616.


\end{document}

%% file: figure1.tex
\coordinate (A1) at (-.951,.309);
			\coordinate (A2) at (-.588,-.809);
			\coordinate (A3) at (.588,-.809);
			\coordinate (A4) at (.951,.309);
			\coordinate (A5) at (0,1);

			\coordinate (B1) at (-1.902,.618);
			\coordinate (B2) at (-1.176,-1.618);
			\coordinate (B3) at (1.176,-1.618);
			\coordinate (B4) at (1.902,.618);
			\coordinate (B5) at (0,2);

			\coordinate (AA1) at ($(4,2)+(170:.8)$);
			\coordinate (AA2) at ($(4,2)+(120:.3)$);
			\coordinate (AA3) at ($(4,2)+(-77:.6)$);
			
			\coordinate (BB1) at ($(4,-2)+(20:.9)$);
			\coordinate (BB2) at ($(4,-2)+(120:.3)$);
			\coordinate (BB3) at ($(4,-2)+(-77:.2)$);

			\draw[thick,blue] (AA1) -- (BB1);
			\draw[thick,blue] (AA2) -- (BB2);
			\draw[thick,blue] (AA3) -- (BB3);
			\draw[thick,blue] (AA1) -- (BB2);
			\draw[thick,blue] (AA3) -- (BB1);

			\draw[thick,red] (AA1) -- (B1);
			\draw[thick,red] (AA2) -- (A2);
			\draw[thick,red] (A3) -- (BB3);
			\draw[thick,red] (A1) -- (BB3);
			\draw[thick,red] (AA3) -- (B2);
			\draw[thick,red] (AA3) -- (B1);
			\draw[thick,red] (AA2) -- (B4);
			\draw[thick,red] (B5) -- (BB3);
			\draw[thick,red] (B2) -- (BB3);
			\draw[thick,red] (AA3) -- (B5);
			
			\draw[fill=white] (0,0) circle (2.3) node at (135:2.5) {$H$};
			\draw[fill=white] (4,2) circle (1.1) node at ($(4,2)+(45:1.3)$) {$A$};
			\draw[fill=white] (4,-2) circle (1.1) node at ($(4,-2)+(-45:1.3)$) {$B$};
			
			\draw[thick,blue] (A1) -- (A3);
			\draw[thick,blue] (A4) -- (A2);
			\draw[thick,blue] (A3) -- (A5);
			\draw[thick,blue] (A1) -- (A4);
			\draw[thick,blue] (A5) -- (A2);
			\draw[thick,blue] (B1) -- (B2);
			\draw[thick,blue] (B2) -- (B3);
			\draw[thick,blue] (B3) -- (B4);
			\draw[thick,blue] (B4) -- (B5);
			\draw[thick,blue] (B5) -- (B1);
			\draw[thick,blue] (A1) -- (B1);
			\draw[thick,blue] (A2) -- (B2);
			\draw[thick,blue] (A3) -- (B3);
			\draw[thick,blue] (A4) -- (B4);
			\draw[thick,blue] (A5) -- (B5);

%% file: figure2.tex
			\begin{scope}[shift={(0,2)}]
						\coordinate (A1) at (-.951,.309);
						\coordinate (A2) at (-.588,-.809);
						\coordinate (A3) at (.588,-.809);
						\coordinate (A4) at (.951,.309);
						\coordinate (A5) at (0,1);

						\coordinate (B1) at (-1.902,.618);
						\coordinate (B2) at (-1.176,-1.618);
						\coordinate (B3) at (1.176,-1.618);
						\coordinate (B4) at (1.902,.618);
						\coordinate (B5) at (0,2);
			\end{scope}

			\coordinate (A11) at ($(6,0)+(90:2) + (-50:.25)$);
			\coordinate (A12) at ($(6,0)+(90:2) + (50:.2)$);
			\coordinate (A13) at ($(6,0)+(90:2) + (133:.1)$);
			\coordinate (A21) at ($(6,0)+(162:2) + (-5:.4)$);
			\coordinate (A22) at ($(6,0)+(162:2) + (70:.2)$);
			\coordinate (A23) at ($(6,0)+(162:2) + (123:.25)$);
			\coordinate (A31) at ($(6,0)+(-126:2) + (166:.4)$);
			\coordinate (A32) at ($(6,0)+(-126:2) + (-98:.3)$);
			\coordinate (A33) at ($(6,0)+(-126:2) + (-66:.33)$);
			\coordinate (A34) at ($(6,0)+(-126:2) + (20:.2)$);
			\coordinate (A41) at ($(6,0)+(-54:2) + (133:.25)$);
			\coordinate (A42) at ($(6,0)+(-54:2) + (-77:.4)$);
			\coordinate (A43) at ($(6,0)+(-54:2) + (1:.47)$);
			\coordinate (A51) at ($(6,0)+(18:2) + (-50:.25)$);
			\coordinate (A52) at ($(6,0)+(18:2) + (166:.3)$);
			\coordinate (A53) at ($(6,0)+(18:2) + (-44:.1)$);
			
			\coordinate (N1) at ($(-1,-2)$);
			\coordinate (N2) at ($(1,-2)$);
			
			\coordinate (N11) at ($(N1)+(.1,.5)$);
			\coordinate (N12) at ($(N1)+(-.4,.2)$);
			\coordinate (N13) at ($(N1)+(.4,-.4)$);
			
			\coordinate (N21) at ($(N2)+(.1,.5)$);
			\coordinate (N22) at ($(N2)+(-.4,.4)$);
			\coordinate (N23) at ($(N2)+(.4,-.5)$);

			\draw[thick,orange] (A1) -- (A31);
			\draw[thick,orange] (A4) -- (A42);
			\draw[thick,orange] (A5) -- (A13);
			\draw[thick,orange] (A3) -- (A22);
			\draw[thick,orange] (A4) -- (A34);
			\draw[thick,orange] (B1) -- (A52);
			\draw[thick,orange] (B3) -- (A51);
			\draw[thick,orange] (B2) -- (A43);
			\draw[thick,orange] (B1) -- (A33);
			\draw[thick,orange] (B4) -- (A21);
			\draw[thick,orange] (A1) -- (A32);
			
			\draw[thick,orange] (N21) -- (A31);
			\draw[thick,orange] (N22) -- (A42);
			\draw[thick,orange] (N21) -- (A13);
			\draw[thick,orange] (N23) -- (A22);
			\draw[thick,orange] (N22) -- (A34);
			\draw[thick,orange] (N21) -- (A52);
			
			\draw[fill=white,draw=white] (6,0) circle (2.3) node at (135:2.5) {$H$};
			
			\draw[thick,blue] (A11) -- (A21);
			\draw[thick,blue] (A11) -- (A22);
			\draw[thick,blue] (A23) -- (A31);
			\draw[thick,blue] (A21) -- (A31);
			\draw[thick,blue] (A22) -- (A33);
			\draw[thick,blue] (A34) -- (A41);
			\draw[thick,blue] (A33) -- (A42);
			\draw[thick,blue] (A32) -- (A41);
			\draw[thick,blue] (A41) -- (A52);
			\draw[thick,blue] (A42) -- (A51);
			\draw[thick,blue] (A51) -- (A11);
			\draw[thick,blue] (A53) -- (A11);
			\draw[thick,blue] (A52) -- (A12);
			
			\draw[thick,blue] (N12) -- (N22);
			\draw[thick,blue] (N11) -- (N23);
			\draw[thick,blue] (N11) -- (N22);
			\draw[thick,blue] (N12) -- (N23);
			\draw[thick,blue] (N13) -- (N21);
			\draw[thick,blue] (N13) -- (N13);
			
			\draw[thick,red] (A1) -- (N22);
			\draw[thick,red] (N11) -- (B3);
			\draw[thick,red] (A4) -- (N22);
			\draw[thick,red] (N12) -- (B2);
			\draw[thick,red] (A2) -- (N21);
			\draw[thick,red] (N13) -- (B5);
			\draw[thick,red] (B2) -- (N22);
			\draw[thick,red] (N11) -- (A3);
			\draw[thick,red] (B5) -- (N22);
			\draw[thick,red] (N12) -- (A2);
			\draw[thick,red] (B1) -- (N21);
			\draw[thick,red] (N13) -- (A1);			

			\draw[thick,red] (A11) -- (A31);
			\draw[thick,red] (A11) -- (A32);
			\draw[thick,red] (A23) -- (A41);
			\draw[thick,red] (A21) -- (A41);
			\draw[thick,red] (A22) -- (A43);
			\draw[thick,red] (A34) -- (A51);
			\draw[thick,red] (A33) -- (A52);
			\draw[thick,red] (A32) -- (A51);
			\draw[thick,red] (A41) -- (A12);
			\draw[thick,red] (A42) -- (A11);
			\draw[thick,red] (A51) -- (A21);
			\draw[thick,red] (A53) -- (A21);
			\draw[thick,red] (A52) -- (A22);

			\begin{scope}[shift={(0,2)}]
				\draw[fill=white] (0,0) circle (2.3) node at (135:2.5) {$H$};
			\end{scope}

			\draw[fill=white] ($(6,0)+(90:2)$) circle (.5) node at ($(6,0)+(90:2)+(90:.7)$) {$A_1$};
			\draw[fill=white] ($(6,0)+(162:2)$) circle (.5) node at ($(6,0)+(162:2)+(162:.7)$) {$A_2$};
			\draw[fill=white] ($(6,0)+(-126:2)$) circle (.5) node at ($(6,0)+(-126:2)+(-126:.7)$) {$A_3$};
			\draw[fill=white] ($(6,0)+(-54:2)$) circle (.5) node at ($(6,0)+(-54:2)+(-54:.7)$) {$A_4$};
			\draw[fill=white] ($(6,0)+(18:2)$) circle (.5) node at ($(6,0)+(18:2)+(18:.7)$) {$A_5$};
			
			\draw[fill=white] ($(N1)$) circle (.8) node at ($(N1)+(-126:1.2)$) {$B_1$};
			\draw[fill=white] ($(N2)$) circle (.8) node at ($(N2)+(-54:1.2)$) {$B_2$};

			\draw[thick,blue] (A1) -- (A3);
			\draw[thick,blue] (A4) -- (A2);
			\draw[thick,blue] (A3) -- (A5);
			\draw[thick,blue] (A1) -- (A4);
			\draw[thick,blue] (A5) -- (A2);
			\draw[thick,blue] (B1) -- (B2);
			\draw[thick,blue] (B2) -- (B3);
			\draw[thick,blue] (B3) -- (B4);
			\draw[thick,blue] (B4) -- (B5);
			\draw[thick,blue] (B5) -- (B1);
			\draw[thick,blue] (A1) -- (B1);
			\draw[thick,blue] (A2) -- (B2);
			\draw[thick,blue] (A3) -- (B3);
			\draw[thick,blue] (A4) -- (B4);
			\draw[thick,blue] (A5) -- (B5);